\documentclass[12pt]{amsart} 
\usepackage{amscd}
\usepackage{amssymb}
\usepackage{a4wide}
\usepackage{amstext}
\usepackage{amsthm}
\usepackage{xcolor}
\numberwithin{equation}{section}

 \DeclareMathOperator{\dist}{dist}

\renewcommand{\phi}{\varphi}

\newtheorem{Thm}{Theorem}[section]
\newtheorem{theorem}[Thm]{Theorem}
\newtheorem{question}{Question}

\newtheorem{lemma}[Thm]{Lemma}

\begin{document}
\sloppy
\title[Uniqueness of Gabor series]{Uniqueness of Gabor series}
\author{Yurii Belov}
\address{Yurii Belov,
\newline Chebyshev Laboratory, St.~Petersburg State University, St. Petersburg, Russia,
\newline {\tt j\_b\_juri\_belov@mail.ru}
}
\thanks{Author was supported by RNF grant 14-21-00035.}

\begin{abstract} We prove that any complete and minimal Gabor system of Gaussians is a Markushevich basis in $L^2(\mathbb{R})$.
\end{abstract}


\maketitle

\section{Introduction}

Let $\Lambda\subset\mathbb{R}^2$ be a sequence of distinct points. With each such sequence we associate Gabor system
\begin{equation}
\mathcal{G}_\Lambda:=\{e^{2\pi i yt}e^{-\pi(t-x)^2}\}_{(x,y)\in\Lambda}.
\end{equation}
Function $e^{2\pi i yt}e^{-\pi(t-x)^2}$ can be viewed  as the time--frequency shift of the Gaussian $e^{-\pi t^2}$ in the phase space. It is well known that system $\mathcal{G}_\Lambda$ cannot be a Riesz basis in $L^2(\mathbb{R})$ (see e.g \cite{S}). On the other hand, there exist a lot of {\it complete and minimal} systems $\mathcal{G}_\Lambda$. A canonical example is the lattice  without one point, $\Lambda:= \mathbb{Z}\times\mathbb{Z}\setminus\{(0,0)\}$. However, the generating sets $\Lambda$ can be very far from any lattice. For example, in \cite{als} it was shown that there exists $\Lambda\subset\mathbb{R}\times\{0\}\cup\{0\}\times\mathbb{R}$ such that $\mathcal{G}_\Lambda$ is complete and minimal in $L^2(\mathbb{R})$.   

If $\mathcal{G}_\Lambda$ is complete and minimal, then there exists the unique biorthogonal system $\{g_{(x,y)}\}_{(x,y)\in\Lambda}$. So, for any $f\in L^2(\mathbb{R})$ we may write the formal Fourier series with respect to the system $\mathcal{G}_\Lambda$
\begin{equation}
f\sim \sum_{(x,y)\in\Lambda}(f,g_{(x,y)})_{L^2(\mathbb{R})}e^{2\pi i yt}e^{-\pi(t-x)^2}.
\label{mainser}
\end{equation}
If $\Lambda =  \mathbb{Z}\times\mathbb{Z}\setminus\{(0,0)\}$, then it is known that there exists a linear summation method for the series \eqref{mainser} (e.g. one can use methods from \cite{ls}). In \cite{ls} this was proved for certain sequences similar to lattices. The main point of the present note is to show that {\it any } series \eqref{mainser} defines an element $f$  uniquely.  
\begin{theorem}
Let $\mathcal{G}_\Lambda$ be a complete and minimal system in $L^2(\mathbb{R})$. Then the biorthogonal system $\{g_{(x,y)}\}_{(x,y)\in\Lambda}$ is complete. So, any function $f\in L^2(\mathbb{R})$ is uniquely determined by the coefficients $(f,g_{(x,y)})$. 
\label{mainth}
\end{theorem}

This property is by no means automatic for an arbitrary system of vectors. Indeed, if $\{e_n\}^\infty_{n=1}$ is an orthonormal basis in a separable Hilbert space, then $\{e_1+e_n\}^\infty_{n=2}$ is a complete and minimal system but its biorthogonal $\{e_n\}^\infty_{n=2}$ is not complete. A complete and minimal system in a Hilbert space with complete biorthogonal system is called {\it Markushevich basis}. 

Theorem \ref{mainth} is analogous to Young's theorem \cite{Young} for systems of complex exponentials $\{e^{i\lambda_n t}\}$ in $L^2$ of an interval. However, the structure of complete and minimal systems for Gabor systems is more puzzling than for the systems of exponentials on an interval. For example, if $\Lambda$ generates a complete and minimal system of exponentials in $L^2(-\pi,\pi)$, then the upper density of $\Lambda$ ($=\lim_{r\rightarrow\infty}\#(\Lambda\cap\{|\lambda|<r\})(2r)^{-1}$) is equal to $1$; see Theorem 1 in Lecture 17 of \cite{Lev}. On the other hand,
if $\mathcal{G}_\Lambda$ is a complete and minimal Gabor system, then the upper density of $\Lambda$ ( $=\lim_{r\rightarrow\infty}\#(\Lambda\cap\{x^2+y^2\leq r^2\})(\pi r^2)^{-1}$) can vary from $2\slash\pi$ to $1$; see Theorem 1 in \cite{als}. 
If, in addition, $\Lambda$ is a regular distributed set, then the upper density have to be from $2\slash\pi$ to $1$; see Theorem 2 in \cite{als}.

Note that for some systems of special functions (associated to some canonical system of differential equations) in $L^2$ of an interval completeness of biorthogonal system may fail (even with infinite defect); see \cite[Proposition 3.4]{bb}.

In the next section we transfer our problem to Fock space of entire functions. The last section is devoted to the proof of our result.

\textbf{Notations.} Throughout this paper the notation $U(x)\lesssim V(x)$ means that there is a constant $C$ such that
$U(x)\leq CV(x)$ holds for all $x$ in the set in question, $U, V\geq 0$. We write $U(x)\asymp V(x)$ if both $U(x)\lesssim V(x)$ and
$V(x)\lesssim U(x)$. 

\section{Reduction to a Fock space problem}

Let
$$\mathcal{F}:=\{F\text{ is entire  and } \int_\mathbb{C}|F(z)|^2e^{-\pi|z|^2}dm(z)<\infty\};$$
here $dm$ denotes the planar Lebesgue measure. It is well known that the following Bargmann transform
$$\mathcal{B}f(z):=2^{1\slash4}e^{-i\pi xy}e^{\frac{\pi}{2}|z|^2}\int_\mathbb{R}f(t)e^{2\pi i yt}e^{-\pi (t-x)^2}dt$$ $$=
2^{1\slash4}\int_{\mathbb{R}}f(t)e^{-\pi t^2}e^{2\pi tz}e^{-\frac{\pi}{2} z^2}dt, \quad z=x+iy,$$
 is a unitary map between $L^2(\mathbb{R})$ and the Fock space $\mathcal{F}$; see \cite{Fol, Gro} for the details.

Moreover, the time--frequency shift of the Gaussian is mapped to the normalized reproducing kernel of $\mathcal{F}$ 
\begin{equation}
2^{1\slash4}\mathcal{B}(e^{2\pi i  ut}e^{-\pi(t-v)^2})(z)=e^{-\pi|w|^2\slash2 }e^{\pi\overline{w}z}= \frac{k_w(z)}{\|k_w\|_\mathcal{F}},\quad w = u- iv,\quad k_w(z):=e^{\pi\bar{w}z}.
\end{equation}
The existence of such transformation allows us to apply methods from the theory of entire functions. For that reason the results about  time--frequency shifts of the Gaussians are stronger than for the time--frequency shifts of other elements of $L^2(\mathbb{R})$. 

\begin{lemma} The system $\mathcal{G}_\Lambda$ is complete and minimal in $L^2(\mathbb{R})$ if and only if the system of reproducing kernels $\bigl{\{}\frac{k_\lambda(z)}{\|k_\lambda\|}\bigr{\}}_{\lambda\in\Lambda}$ is complete and minimal in $\mathcal{F}$.
\label{l1}
\end{lemma}
\begin{proof} 
The system $\mathcal{G}_\Lambda$ is complete and minimal if and only if the system  $\mathcal{G}_{\overline{\Lambda}}$ is complete and minimal.
Now Lemma \ref{l1} immediately follows from the unitarity of Bargmann transform. 
\end{proof}
In many spaces of entire functions the system biorthogonal to the system of reproducing kernels can be described via the generating function; see e.g. Theorem 4 in Lecture 18 of \cite{Lev} (this idea goes back to Paley and Wiener).
\begin{lemma}
The system $\{k_\lambda\}_{\lambda\in\Lambda}$ is complete and minimal in $\mathcal{F}$ if and only if there exists an entire function $F$ such that
$F$ has simple zeros exactly at $\Lambda$, $\frac{F(z)}{z-\lambda}$ belongs to $\mathcal{F}$ for some (any) $\lambda\in\Lambda$ and there is no non-trivial entire function $T$ such that $FT\in\mathcal{F}$.
\label{l2}
\end{lemma}
\begin{proof}
{\it Necessity.}The system $\{k_\lambda\}_{\lambda\in\Lambda}$ has a biorthogonal system which we will call $\{F_\lambda\}_{\lambda\in\Lambda}$. We know that 
$F_{\lambda_1}(z)\frac{z-\lambda_1}{z-\lambda_2}\in\mathcal{F}$ for any $\lambda_1,\lambda_2\in\Lambda$. This function vanishes at the points $\lambda\in \Lambda\setminus\{\lambda_2\}$ and so it equals $F_{\lambda_2}$ up to a multiplicative constant. Hence, the function $c_\lambda F_\lambda(z)(z-\lambda)$ does not depend on $\lambda$ for suitable coefficients $c_\lambda$. Denote it by $F$. It is easy to see that $F$ satisfies the required properties.

{\it Sufficiency.} Assume that such $F$ exists. From the inclusion $\frac{F(z)}{z-\lambda_0}\in\mathcal{F}$ we conclude that the system $\{k_\lambda\}_{\lambda\in\Lambda\setminus\{\lambda_0\}}$ is not complete. On the other hand, if the whole system $\{k_\lambda\}_{\lambda\in\Lambda}$ is not complete, then there exists $T$ such that $FT\in\mathcal{F}$.
\end{proof}
The function $F$ from Lemma \ref{l2} is called {\it a generating function} of $\Lambda$. So, the following theorem is the reformulation of Theorem \ref{mainth} in terms of the Fock space.
\begin{theorem}
If $\{k_\lambda\}$ is a complete and minimal system of reproducing kernels in $\mathcal{F}$ and $F$ is the generating function of this system, then the system $\bigl{\{}\frac{F(z)}{z-\lambda}\bigr{\}}_{\lambda\in\Lambda}$ is also complete.
\label{fokth}
\end{theorem} 
In the last section we will prove this theorem.

\section{Completeness of biorthogonal system}

\subsection{Preliminary steps} Let $\sigma$ be the Weierstrass $\sigma$-function associated to the lattice $\mathcal{Z}=\{z: z=m+in, m,n\in\mathbb{Z}\}$,
$$\sigma(z)=z\prod_{\lambda\in\mathcal{Z}\setminus\{0\}}\biggl{(}1-\frac{z}{\lambda}\biggr{)}e^{\frac{z}{\lambda}+\frac{z^2}{\lambda^2}}.$$
 It is well known that $|\sigma(z)|\asymp\dist(z,\mathcal{Z})e^{\pi|z|^2\slash2}$; see e.g. \cite[p. 108]{S2}. From this estimate it is easy to see that system $\bigl{\{}\frac{k_w}{\|k_w\|}\bigr{\}}_{w\in\mathcal{Z}\setminus\{0\}}$ is a complete and minimal system and $\sigma_0(z):=\frac{\sigma(z)}{z}$ is its generating function. The system
 $\bigl{\{}\frac{\|k_w\|}{\sigma'_0(w)}\cdot\frac{\sigma_0(z)}{z-w}\bigr{\}}$ is  the biorthogonal system.
With any function $S\in\mathcal{F}$ we can associate its formal Fourier series {\it with respect to the system } $\bigl{\{}\frac{k_w}{\|k_w\|}\bigr{\}}_{w\in\mathcal{Z}\setminus\{0\}}$ 
$$S\sim \sum_{w\in \mathcal{Z}\setminus\{0\}}b_w\frac{k_w}{\|k_w\|},\quad\qquad b_w:=
\biggl{(}S(z),\frac{\|k_w\|}{\sigma'_0(w)}\cdot\frac{\sigma_0(z)}{z-w}\biggr{)}_\mathcal{F}.$$
 This series is more regular than an arbitrary Fourier series \eqref{mainser}. For example this series admits a linear summation method. In particular, we know that the sequence $\{b_w\}$ is non-trivial.
 We need the following straightforward estimate of coefficients
\begin{equation}
|b_w|^2\leq \|S\|^2\cdot\biggl{\|}\frac{\|k_w\|}{\sigma'_0(w)}\cdot\frac{\sigma_0(z)}{z-w}\biggr{\|}^2\lesssim
\|S\|^2\cdot \biggl{\|}\frac{w\sigma_0(z)}{z-w}\biggr{\|}^2$$
$$\lesssim\|S\|^2\cdot\biggl{[}\int_{|z|<2|w|}|\sigma^2_0(z)|e^{-\pi|z|^2}dm(z)+1\biggr{]}\lesssim\|S\|^2\cdot\log(1+|w|).
\label{best}
\end{equation}
\begin{lemma}
If $F$ is the generating function of a complete and minimal system of reproducing kernels $\{k_\lambda\}_{\lambda\in\Lambda}$ in $\mathcal{F}$ and $\Lambda\cap\mathcal{Z}=\emptyset$ , then for any triple $\lambda_1,\lambda_2,\lambda_3\in\Lambda$ we have
\begin{equation}
\biggl{(}\frac{F(z)}{(z-\lambda_1)(z-\lambda_2)(z-\lambda_3)}, S\biggr{)}_{\mathcal{F}}=\sum_{w\in\mathcal{Z}\setminus\{0\}}\frac{F(w)b_w}{(w-\lambda_1)(w-\lambda_2)(w-\lambda_3)\|k_w\|}
\label{intform}
\end{equation}
for any $S\in\mathcal{F}$.
\label{l3}
\end{lemma}
\begin{proof} It is well known that for any function $H\in\mathcal{F}$ we have $\sum_{w\in\mathbb{Z}\setminus\{0\}}\frac{|H(w)|^2}{\|k_w\|^2}<\infty$ (see e.g. \cite{BMO}). So, $\bigl{\{}\frac{F(w)}{(w-\lambda_1)\|k_w\|}\bigr{\}}\in\ell^2$. From \eqref{best} we conclude that $\bigl{\{}\frac{b_w}{(w-\lambda_2)(w-\lambda_3)}\bigr{\}}\in\ell^2$. Hence, the series on the right hand side of \eqref{intform} converges and defines a bounded linear functional on $\mathcal{F}$. On the other hand, the left hand side and the right hand side of \eqref{intform} coincides if $S$ is a finite linear combination of $\{k_w\}_{w\in\mathcal{Z}\setminus\{0\}}$.
\end{proof}

\subsection{Proof of Theorem \ref{mainth}} Assume the contrary. Then there exists a function $S\in\mathcal{F}$ such that $S\perp\frac{F(z)}{z-\lambda}$ for any $\lambda\in\Lambda$.  Without loss of generality we can assume that $\Lambda\cap\mathcal{Z}=\emptyset$.  From the identity
$$\frac{1}{(z-\lambda_1)(z-\lambda_2)(z-\lambda_3)}=\sum_{k=1}^3\frac{c_k}{z-\lambda_k}$$
we get that for any triple $\lambda_1,\lambda_2,\lambda_3\in\Lambda$ 
$$\biggl{(}\frac{F(z)}{(z-\lambda_1)(z-\lambda_2)(z-\lambda_3)}, S\biggr{)}=0.$$
Fix two arbitrary points $\lambda_1,\lambda_2\in\Lambda$. 
Put 
$$L(z)=\sum_{w\in\mathcal{Z}\setminus\{0\}}\frac{F(w)\overline{b_w}}{(z-w)(w-\lambda_1)(w-\lambda_2)\|k_w\|}.$$
Using Lemma \ref{l3} we get that meromorphic function $L$ vanishes at $\Lambda\setminus\{\lambda_1,\lambda_2\}$. Hence,
\begin{equation}
\sum_{w\in\mathcal{Z}\setminus\{0\}}\frac{F(w)\overline{b_w}}{(z-w)(w-\lambda_1)(w-\lambda_2)\|k_w\|}=\frac{F(z)T(z)}{(z-\lambda_1)(z-\lambda_2)\sigma_0(z)},
\label{leq}
\end{equation}
where $T$ is some non-zero entire function. Comparing the residues of both sides of \eqref{leq} we get $T(w)=\overline{b_w}\frac{\sigma'_0(w)}{\|k_w\|}$, $w\in\mathcal{Z}\setminus\{0\}$. 
Assume that $T$ has at least two zeros $t_1,t_2$, then
\begin{equation}
F(z)\frac{T(z)}{(z-t_1)(z-t_2)}=\sum_{w\in\mathcal{Z}\setminus\{0\}}\frac{|w|^{1\slash2}\sigma_0(z)}{z-w}\cdot\frac{F(w)\overline{b_w}}{(w-t_1)(w-t_2)|w|^{1\slash2}\|k_w\|}.
\label{FT}
\end{equation}
From the inclusion $\bigl{\{}\frac{F(w)}{(w-\lambda_1)\|k_w\|}\bigr{\}}\in\ell^2$ and estimates $|b_w|^2\lesssim\log(1+|w|)$,
$\biggl{\|}\frac{|w|^{1\slash2}\sigma_0(z)}{z-w}\biggr{\|}\lesssim1$ we get that the right hand side of \eqref{FT}  belongs to $\mathcal{F}$. This contradicts the completeness of sequence $\{k_\lambda\}_{\lambda\in\Lambda}$.

Hence $T$ has at most one zero. So, $T(z)=e^{P(z)}(a_1z-a_0)$, where $P$ is a polynomial of degree at most $2$. This contradicts the estimate $|T(w)|=\bigl{|}b_w\frac{\sigma'_0(w)}{\|k_w\|}\bigr{|}\lesssim\frac{\log^{1\slash2}(1+|w|)}{|w|}$, $w\in\mathcal{Z}\setminus\{0\}$.
\qed

\subsection{Concluding remarks}

1. The author wonders if the following
statement (stronger that Theorem \ref{mainth}) is true:

\begin{question}
Any complete and minimal Gabor system is a strong Markushevich basis. Which means that any vector $f\in L^2(\mathbb{R})$ belongs to the closed linear span of members of its Fourier series \eqref{mainser} (see \cite{BBB} and references therein).
\end{question}

For systems of complex exponentials $\{e^{i\lambda_n t}\}$ in $L^2$ of an interval this is not true; see \cite[Theorem 2]{BBB}.

\medskip

2. Using our methods one can prove the completeness of the system $\{\frac{F(z)}{z-\lambda}\}_{\lambda\in\Lambda}$ under weaker assumptions that in Theorem \ref{fokth} (e.g if $F\in \mathcal{F}$ and $z^nF\notin\mathcal{F}$, $n\in\mathbb{N}$). Nevertheless we prefer to formulate the result as it is to avoid inessential technicalities.

\medskip

{\bf Acknowledgements.}
 A part of the present work was done when author was
visiting Norwegian University of Science and Technology, whose hospitality is greatly appreciated. 


\end{document}